\documentclass[12pt,reqno]{amsart}

\usepackage[latin1]{inputenc}
\usepackage{amsmath}
\usepackage{amsfonts}
\usepackage{amssymb}
\usepackage{graphics}
\usepackage{enumerate}
\usepackage{amssymb,amsmath,amsthm,amscd,latexsym,verbatim,graphicx,amsfonts}

\usepackage{amscd}
\usepackage{amsmath}
\usepackage{amssymb}
\usepackage{amsthm}
\usepackage{latexsym}
\usepackage{verbatim}

\theoremstyle{plain}

\newtheorem{theorem}{Theorem}[section]

\newtheorem{corollary}[theorem]{Corollary}
\newtheorem{lemma}[theorem]{Lemma}
\newtheorem{prop}[theorem]{Proposition}
\theoremstyle{definition}

\theoremstyle{remark}

\newcommand{\bbR}{\mathbb{R}}
\newcommand{\bbC}{\mathbb{C}}

\newcommand{\bbD}{\mathbb{D}}
\newcommand{\bbN}{\mathbb{N}}

\newcommand{\mco}{\mathcal{O}}

\newcommand{\eitheta}{e^{i\theta}}

\newcommand{\mutil}{\tilde{\mu}}

\newcommand{\dtpi}{\frac{d\theta}{2\pi}}

\newcommand{\nri}{n\rightarrow\infty}
\newcommand{\blri}{L\rightarrow\infty}
\newcommand{\kri}{k\rightarrow\infty}

\DeclareMathOperator*{\supp}{supp}
\DeclareMathOperator*{\Real}{Re}
\DeclareMathOperator*{\Imag}{Im}

\begin{document}
\title[ ] {Two Universality Results for Polynomial Reproducing Kernels}

\bibliographystyle{plain}

\thanks{  }


\maketitle




\begin{center}
\textbf{Brian Simanek}\\
\textit{\small{Baylor University Department of Mathematics\\
One Bear Place $\#$97328, Waco, TX 76798}}\\
\small{Brian$\_$Simanek@baylor.edu}
\end{center}


\begin{abstract}
We prove two new universality results for polynomial reproducing kernels of compactly supported measures.  The first applies to measures on the unit circle with a jump and a singularity in the weight at $1$ and the second applies to area-type measures on a certain disconnected polynomial lemniscate.  In both cases, we apply methods developed by Lubinsky to obtain our results.
\end{abstract}

\vspace{5mm}


\smallskip


\section{Introduction}\label{intro}

\subsection{Background and Results}  Given a finite, positive, and compactly supported measure $\mu$ with infinitely many points in its support, let $\{\varphi_n(z)\}_{n=0}^{\infty}$ be the corresponding sequence of orthonormal polynomials satisfying
\[
\int\varphi_n(z)\overline{\varphi_m(z)}d\mu(z)=\delta_{m,n}.
\]
The leading coefficient of $\varphi_n$ is $\kappa_n$ and $\varphi_n/\kappa_n$ is a monic polynomial, which we denote by $\Phi_n$.  If ever it is necessary to specify the measure of orthogonality, we will write $\varphi_n(z;\mu)$, $\Phi_n(z;\mu)$, and $\kappa_n(\mu)$.  The degree $n$ polynomial reproducing kernel $K_n$ is given by
\[
K_n(z,w;\mu):=\sum_{m=0}^n\varphi_m(z)\overline{\varphi_m(w)}
\]
and is so named because if $Q(z)$ is a polynomial of degree at most $n$, then
\[
\int Q(z)K_n(w,z;\mu)d\mu(z)=Q(w).
\]

When one speaks of universality limits for such kernels, one is interested in determining existence of the limit
\begin{align}\label{klim}
\lim_{\nri}\frac{K_n(z+\epsilon_1(n),z+\epsilon_2(n);\mu)}{K_n(z,z;\mu)},
\end{align}
where $\epsilon_j(n)\rightarrow0$ as $\nri$ in a specific way for $j=1,2$.  The motivation for calculating such limits comes from random matrix theory and we refer the reader to \cite{Deift,LubMich} for further details.  The term ``universality" is used when one can establish existence of the limit (\ref{klim}) for a large class of measures $\mu$ and points $z\in\supp(\mu)$ in such a way that the limiting expression is independent of the measure $\mu$.  An example of such a universality limit is the sine kernel asymptotics that hold for a large class of well-behaved measures on the unit circle and real line (see for example \cite{Breuer,Findley,NewUniv,LubNg,SimExt,To09,To16}).  Many methods used to derive such asymptotics are described in \cite{Lub09}.  Another example is the Bessel kernel asymptotics that are often observed at the edge of the support of a well-behaved measure on the real line (see \cite{Danka2,KVL,LeviLubEdge,LubEdge,LubEdge2}).

A key aspect in all of the aforementioned results is the so-called \textit{regularity} of the measure in question.  We follow the terminology introduced in \cite{StaTo} and say that a measure is regular if
\[
\lim_{\nri}\kappa_n^{1/n}=\frac{1}{\mbox{cap}(\supp(\mu))},
\]
where $\mbox{cap}(K)$ is the logarithmic capacity of the compact set $K$.  The consequences of regularity are complicated to state without much additional notation, so we will restrict our attention to discussing these consequences in the contexts of the specific supports on which we will focus.  For now, we mention that a measure $\mu$ on the unit circle is regular if and only if
\begin{equation}\label{regprop}
\lim_{\nri}\left(\sup_{\deg(P)\leq n}\left[\frac{\|P\|_{L^{\infty}(\partial\bbD)}}{\|P\|_{L^2(\mu)}}\right]^{1/n}\right)=1
\end{equation}
(see \cite[Theorem 3.2.3]{StaTo}).

Universality limits for less well-behaved measures are more difficult to establish.  The 2011 paper \cite{FMMFS} established a universality result for measures on $[-1,1]$ with a jump in the weight at $0$ and the limiting kernel is expressed in terms of the confluent hypergeometric function $_1F_1$ (see also \cite{IK}).  The recent paper \cite{Danka2} considers measures on the real line that have a certain singular behavior near an interior point in the support of the measure.  In 2012, Bourgade \cite[Theorem 3.2]{Bourgade} established a result on the unit circle when the measure displays a certain singularity at $1$ (see also \cite[Theorem 5]{BNR}).  Our first theorem is a generalization of Bourgade's result.  For the statement of the theorem and throughout this paper, we identify the unit circle with the interval $[0,2\pi]$.

\begin{theorem}\label{singjump}
Let $\mu$ be a regular measure on the unit circle given by $w(\theta)\dtpi+d\mu_s$, where $\mu_s$ is singular with respect to Lebesgue measure.  Suppose
\[
w(\theta)=g(\theta)\frac{4^{\gamma}|\Gamma(1+\gamma+i\tau)|^2}{\Gamma(2\gamma+1)}e^{(\pi-\theta)\tau}[\sin(\theta/2)]^{2\gamma},\qquad\theta\in[0,2\pi],\quad \tau\in\bbR,\quad \gamma>-\frac{1}{2},
\]
where $g:\bbR\rightarrow[0,\infty)$ is $2\pi$-periodic, continuous at $0$ with $g(0)>0$, and such that $w(\theta)$ is integrable on $[0,2\pi]$.  Suppose also that $\{0,2\pi\}\cap\supp(\mu_s)=\emptyset$.  Then uniformly for $a$ and $b$, in compact subsets of $\bbC$ it holds that
\begin{align}\label{kform0}
&\lim_{\nri}\frac{K_n(e^{ia/n},e^{ib/n};\mu)}{K_n(1,1;\mu)}\\
\nonumber&\,=(2\gamma+1)\frac{_1F_1(\bar{y};2\gamma+1;-i\bar{b})_1F_1(y;2\gamma+1;ia)-\,_1F_1(1+\bar{y};2\gamma+1;-i\bar{b})_1F_1(1+y;2\gamma+1;ia)}{i(\bar{b}-a)},
\end{align}
where $y=\gamma+i\tau$.
\end{theorem}

\noindent\textit{Remark.}  If $a=\bar{b}$, then we use continuity to interpret the right-hand side of (\ref{kform0}) as
\begin{align*}
&\bigg(\,_1F_1(1+\bar{y};2\gamma+1;-i\bar{b})\,_1F_1(2+y;2\gamma+2;i\bar{b})(1+y)\\
&\qquad\qquad\qquad\qquad\qquad\qquad\qquad-\,_1F_1(\bar{y};2\gamma+1;-i\bar{b})\,_1F_1(1+y;2\gamma+2;i\bar{b})y\bigg).
\end{align*}

\medskip

\noindent\textit{Remark.}  Notice that if $a$ and $b$ are real and we set $\gamma=\tau=0$ in (\ref{kform0}), then we recover \cite[Theorem 1.1]{LeviLub2}.  For a related result, see \cite[Theorem 5.9]{SY2}.

\smallskip

Although similar, Theorem \ref{singjump} is distinct from Bourgade's result \cite[Theorem 3.2]{Bourgade}.  The most important difference is that Bourgade's result concerns the normalized kernel
\[
\tilde{K}_n(x,y):=\sqrt{w(x)w(y)}\sum_{k=0}^n\varphi_k(x)\overline{\varphi_k(y)},
\]
which is only defined when $x,y\in\partial\bbD$, while our result allows $a$ and $b$ to be complex.  This is especially relevant in determining the local asymptotics of the Christoffel function for complex $a$ (see Corollary \ref{lambdas} below).  Furthermore, if $\gamma<0$, then $\tilde{K}_n$ is undefined if either $x$ or $y$ is equal to $1$ because of the singularity in $w$ there.  Consequently, our Theorem \ref{singjump} includes a stronger uniformity statement than the one given in \cite[Theorem 3.2]{Bourgade}.  Nevertheless, we should mention that our proof of Theorem \ref{singjump} is similar to the proof of \cite[Theorem 3.2]{Bourgade} in that it relies on ideas and methods of Lubinsky (see for example \cite{LubMR}).

As noted in \cite{FMMFS}, the confluent hypergeometric function appears in the scaling limit of correlation functions of the pseudo-Jacobi ensemble in \cite{BO}.  It is also explained in \cite{FMMFS} how one can understand the connection between the sequence of weights in the pseudo-Jacobi ensemble and the weight in Theorem \ref{singjump} (see \cite[Equation 24]{FMMFS} and the accompanying discussion).  The weights considered in \cite{BO} are of the form
\[
(1+x^2)^{-\Real(s)-N}e^{2\Imag(s)\arg(1+ix)},\qquad\qquad x\in\bbR,\qquad N\in\bbN.
\]
If one writes $z=(1+ix)^{-1}$, then this becomes
\[
|z|^{2\Real(s)+2N}e^{2\Imag(s)\arg(1/z)},\qquad\qquad|z-1/2|=1/2,
\]
which has an algebraic singularity and a jump at $0$, just like the weight in our Theorem \ref{singjump} does at $\theta=0$.  Further motivation for the study of measures such as those that we consider in Theorem \ref{singjump} can be found in \cite{Bourgade}.

By setting $\gamma=0$ in Theorem \ref{singjump}, we immediately obtain the following analog of \cite[Theorem 11]{FMMFS} for the unit circle.

\begin{corollary}\label{circlejump}
Let $\mu$ be a regular measure on the unit circle given by $w(\theta)\dtpi+d\mu_s$, where $\mu_s$ is singular with respect to Lebesgue measure.  Suppose
\[
w(\theta)=g(\theta)|\Gamma(i\tau+1)|^2e^{(\pi-\theta)\tau},\qquad\qquad\theta\in[0,2\pi],\quad \tau\in\bbR,
\]
where $g:\bbR\rightarrow[0,\infty)$ is $2\pi$-periodic, integrable on $[0,2\pi]$, and continuous at $0$ with $g(0)>0$.  Suppose also that $\{0,2\pi\}\cap\supp(\mu_s)=\emptyset$.  Then uniformly for $a$ and $b$, in compact subsets of $\bbC$ it holds that
\begin{align}\label{kform1}
&\lim_{\nri}\frac{K_n(e^{ia/n},e^{ib/n};\mu)}{K_n(1,1;\mu)}\\
\nonumber&\qquad\qquad\qquad\qquad=\frac{_1F_1(-i\tau;1;-i\bar{b})_1F_1(i\tau;1;ia)-\,_1F_1(1-i\tau;1;-i\bar{b})_1F_1(1+i\tau;1;ia)}{i(\bar{b}-a)}.
\end{align}
\end{corollary}

\noindent\textit{Remark.}  If $a=\bar{b}$, we interpret the expression on the right-hand side of (\ref{kform1}) as
\[
_1F_1(1-i\tau;1;-i\bar{b})\,_1F_1(2+i\tau;2;i\bar{b})(i\tau+1)-\,_1F_1(-i\tau;1;-i\bar{b})\,_1F_1(1+i\tau;2;i\bar{b})(i\tau)
\]

\medskip

If we set $\tau=0$ in Theorem \ref{singjump}, then we obtain the following result, which is reminiscent of \cite[Theorem 1.4]{Danka2}.

\begin{corollary}\label{circlesing}
Let $\mu$ be a regular measure on the unit circle given by $w(\theta)\dtpi+d\mu_s$, where $\mu_s$ is singular with respect to Lebesgue measure.  Suppose
\[
w(\theta)=g(\theta)\frac{4^{\gamma}|\Gamma(1+\gamma)|^2}{\Gamma(2\gamma+1)}[\sin(\theta/2)]^{2\gamma},\qquad\qquad\theta\in[0,2\pi],\qquad \gamma>-\frac{1}{2},
\]
where $g:\bbR\rightarrow[0,\infty)$ is $2\pi$-periodic, continuous at $0$ with $g(0)>0$, and such that $w(\theta)$ is integrable on $[0,2\pi]$.  Suppose also that $\{0,2\pi\}\cap\supp(\mu_s)=\emptyset$.  Then uniformly for $a$ and $b$, in compact subsets of $\bbC$ it holds that
\begin{align}\label{kform2}
&\lim_{\nri}\frac{K_n(e^{ia/n},e^{ib/n};\mu)}{K_n(1,1;\mu)}\\
\nonumber&\quad=(2\gamma+1)\frac{_1F_1(\gamma;2\gamma+1;-i\bar{b})_1F_1(\gamma;2\gamma+1;ia)-\,_1F_1(1+\gamma;2\gamma+1;-i\bar{b})_1F_1(1+\gamma;2\gamma+1;ia)}{i(\bar{b}-a)}
\end{align}
\end{corollary}

\noindent\textit{Remark.}  If $a=\bar{b}$, we interpret the expression on the right-hand side of (\ref{kform2}) as
\begin{align*}
&\bigg(\,_1F_1(1+\gamma;2\gamma+1;-i\bar{b})\,_1F_1(2+\gamma;2\gamma+2;i\bar{b})(1+\gamma)\\
&\qquad\qquad\qquad\qquad\qquad\qquad\qquad-\,_1F_1(\gamma;2\gamma+1;-i\bar{b})\,_1F_1(1+\gamma;2\gamma+2;i\bar{b})\gamma\bigg).
\end{align*}

\bigskip

The form of the limiting kernel in (\ref{kform1}) could be guessed from the results in \cite{FMMFS}.  One distinction of our results is that the proof does not rely on the methods of Riemann-Hilbert analysis.  Instead we rely on the method that has been pioneered by Lubinsky and has lead to many new results on this topic.



\bigskip

In recent years there has also been interest in proving universality results for measures supported on sets other than the unit circle or real line.  Such results include measures supported on arcs of the unit circle (see \cite{LubNg}), smooth Jordan regions (see \cite{LubBerg,SY}), and smooth Jordan curves (see \cite{LeviLub}).  Our second main result is a comparable result for a collection of polynomial lemniscates.

Given any $r>0$ and $m\in\bbN$, let $G_{r,m}$ be the region given by $G_{r,m}:=\{z:|z^m-1|<r^m\}$.  For any set $X\subseteq\bbC$, we will denote its boundary by $\partial X$.  For a measure $\mu$ on $G_{r,m}$, we will let $\mu'$ denote its Radon-Nikodym derivative with respect to area measure on $G_{r,m}$.

\begin{figure}[h!]\label{Gpic}
\centering
\includegraphics[width=0.52\textwidth]{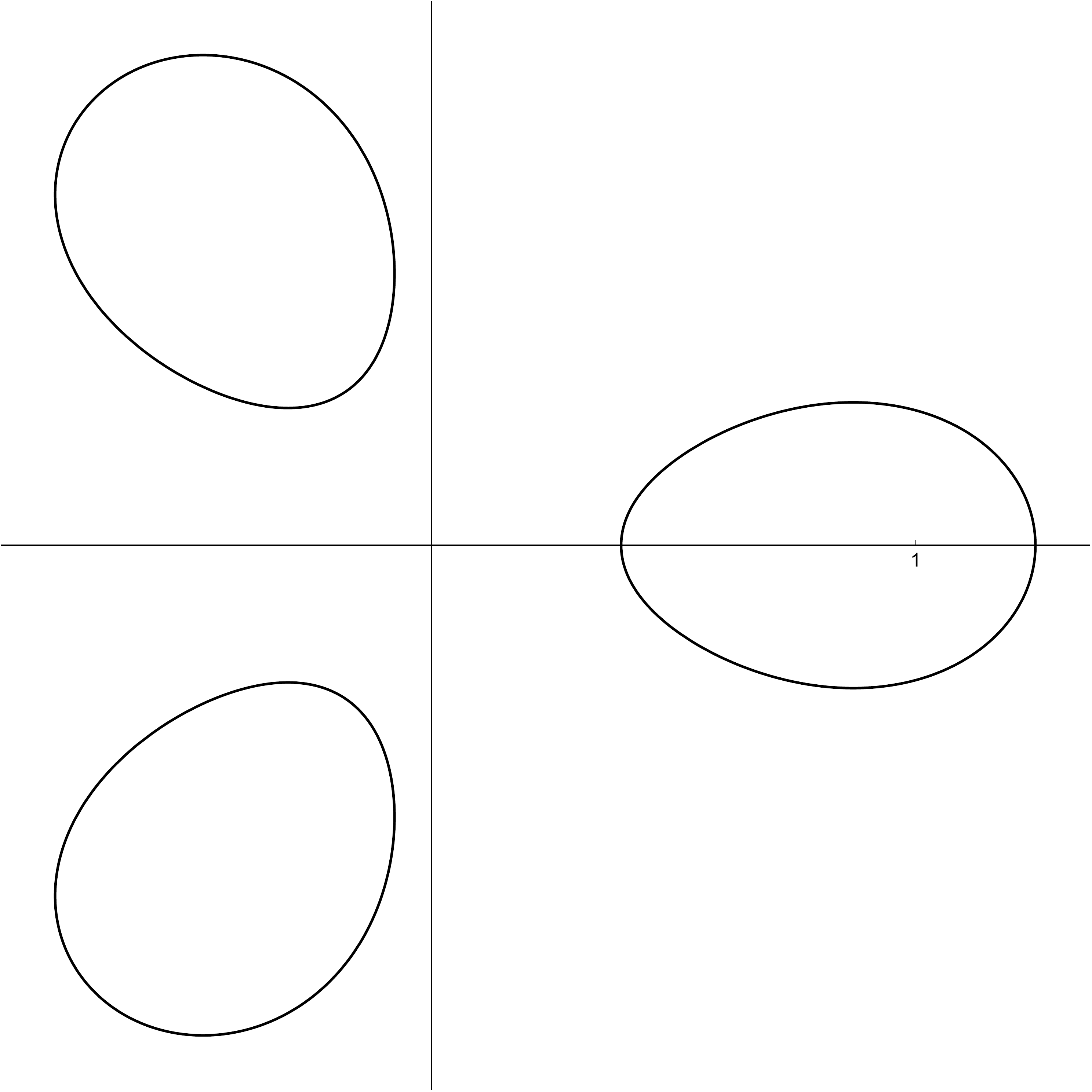}
\caption{A Mathematica plot of the boundary of the region $G_{r,3}$ when $r=\sqrt[3]{0.94}$.}
\end{figure}

\begin{theorem}\label{bergl}
Let $r\in(0,1)$ and $m\in\bbN$ be fixed and pick some $z^*_0$ in the boundary of $G_{r,m}$.  Let $\mu$ be a regular measure on $\overline{G}_{r,m}$ and suppose that there is an open disk $D$ centered at $z_0^*$ such that on $G_{r,m}\cap D$, $\mu$ is absolutely continuous with respect to planar Lebesgue measure.  Assume moreover that $\mu'$ is positive and continuous at each point of $\partial G_{r,m}\cap D$.  Then uniformly for $z_0$ in compact subsets of $\partial G_{r,m}\cap D$ and $a,b$ in compact subsets of $\bbC$ it holds that
\begin{align}\label{kform3}
\lim_{\nri}\frac{K_n(z_0+a/n,z_0+b/n;\mu)}{K_n(z_0,z_0;\mu)}=H\left(\frac{\overline{w}_0az_0^{m-1}+w_0\bar{b}\bar{z}_0^{m-1}}{r^{m}}\right),
\end{align}
where $w_0=\frac{z_0^m-1}{r^m}\in\partial\bbD$ and
\[
H(t)=\,_1F_1(2;3;t)=\begin{cases}
2\frac{e^t(t-1)+1}{t^2} \qquad & t\neq0\\
1 \qquad & t=0.
\end{cases}
\]
\end{theorem}

\noindent\textit{Remark.}  Notice that if we define $\phi(z)=\frac{z^m-1}{r^m}$, then the right-hand side of (\ref{kform2}) becomes
\[
H\left(a\frac{\phi'(z_0)}{m}\overline{\phi(z_0)}+\bar{b}\frac{\overline{\phi'(z_0)}}{m}\phi(z_0)\right).
\]
From this it is clear that if we substitute $m=1$ into (\ref{kform3}), then we recover the result of \cite[Theorem 1.1]{LubBerg} in the case of the disk of radius $r$ centered at $1$.

\smallskip

\noindent\textit{Remark.}  The argument of $a\phi'(z_0)\overline{\phi(z_0)}$ appearing in this expression is the angle between $a$ and the outward normal to $G_{r,m}$ at $z_0$ (see  \cite[Remark 5.1]{SY}).

\smallskip

\noindent\textit{Remark.}  An alternative formula for $H$ is given by
\[
H(t)=2\int_0^1xe^{tx}\,dx,
\]
which means this kernel fits into the family of kernels discussed in \cite[Section 5.4]{SY2}.

\bigskip

All of our proofs will utilize the method of Lubinsky, which relies on Christoffel function estimates.  For a measure $\mu$ and $n\in\bbN$, we define the Christoffel function $\lambda_n(z;\mu)$ by
\[
\lambda_n(z;\mu):=\inf\left\{\|P\|^2_{L^2(\mu)}:\deg(P)\leq n,\, P(z)=1\right\}.
\]
It is well-known and easy to show that
\[
\lambda_n(z;\mu)=\frac{1}{K_n(z,z;\mu)}
\]
and the infimum defining $\lambda_n$ is a minimum with extremal function given by $P(w)=\frac{K_n(w,z;\mu)}{K_n(z,z;\mu)}$.  This relationship tells us that the diagonal of $K_n$ satisfies certain monotonicity properties, which will be helpful in our later analysis.  Specifically, if $\mu\geq\nu$, then $K_n(z,z;\mu)\leq K_n(z,z;\nu)$.  For more about Christoffel functions and their relation to orthogonal polynomials, we refer the reader to \cite{MNT,NevaiFreud}.

\bigskip

The rest of the paper is devoted to the proofs of the results we have already stated.  In Section \ref{circproof} we will prove Theorem \ref{singjump} and in Section \ref{bergproof} we will prove Theorem \ref{bergl}.  In both cases, the key step is establishing the needed asymptotics in an appropriate model case that we will use for comparison.  In the unit circle setting, we will rely on results from \cite{Bourgade,BNR,Sri}.  In the lemniscate setting, we will rely on results from \cite{Islands} and \cite{ToTrans}.

\subsection{Notation}

For any complex number $v$ and $n\in\{0,1,2,\ldots\}$, we let $(v)_n$ denote the rising factorial
\[
(v)_n=\begin{cases}
1 & \qquad n=0\\
v(v+1)\cdots(v+n-1) & \qquad n>0.
\end{cases}
\]
If $\Phi_n(z)$ is the degree $n$ monic orthogonal polynomial for a measure $\mu$, then we define the polynomial $\Phi_n^*$ as in \cite{OPUC1} by
\[
\Phi_n^*(z):=z^n\overline{\Phi_n(1/\bar{z})}.
\]

\section{Proof of Theorem \ref{singjump}}\label{circproof}

In this section, we present a proof of Theorem \ref{singjump}.  To do so, we will first need a comparison example, and for that we turn to the results in \cite{Sri}.

\subsection{A Model Case.}\label{firstc}

We recall the results of \cite{Sri}, which concern the measure $\mutil$ defined by
\[
d\mutil(\theta)=\frac{4^{\gamma}|\Gamma(1+\gamma+i\tau)|^2}{\Gamma(2\gamma+1)}e^{(\pi-\theta)\tau}[\sin(\theta/2)]^{2\gamma}\frac{d\theta}{2\pi},\qquad0\leq\theta\leq2\pi,\quad\tau\in\bbR,\quad \gamma>-\frac{1}{2}.
\]
In \cite{Bourgade}, the measure $\mutil$ is said to have a \textit{Hua-Pickrell density} and in \cite[Section 3.2]{BNR}, the weight is called a \textit{Fisher-Hartwig function}.  It is shown in \cite[Theorem 4.1]{Sri} and the proof of \cite[Theorem 4.2]{Sri} that the monic orthogonal polynomials and the leading coefficients of the orthonormal polynomials are given by
\begin{align*}
\Phi_n(z)&=\frac{(2\gamma+1)_n}{(1+\gamma+i\tau)_n}\,_2F_1(-n,1+\gamma+i\tau;2\gamma+1;1-z),\\
\Phi_n^*(z)&=\frac{(2\gamma+1)_n}{(1+\gamma-i\tau)_n}\,_2F_1(-n,\gamma+i\tau;2\gamma+1;1-z),\\
\kappa_n&=\frac{|(1+\gamma+i\tau)_n|}{\sqrt{n!(2\gamma+1)_n}}
\end{align*}
(see also \cite[Section 3.2.1]{Bourgade}, compare with \cite[Proposition 1.2]{BO}).  With this information, one can deduce the following fact, which appears in the proof of \cite[Theorem 5]{BNR}.



\begin{lemma}\label{kn11}
For the measure $\mutil$, it holds that
\begin{align}
\nonumber K_n(1,1;\mutil)&=\frac{n^{2\gamma+1}}{\Gamma(2\gamma+2)}(1+o(1))
\end{align}
as $\nri$.
\end{lemma}

By applying Lemma \ref{kn11}, we conclude that if $a,b\in\bbR\setminus\{0\}$ and $a\neq b$, then\footnote{Even though we are assuming $a,b\in\bbR$ in this calculation, we write $\bar{b}$ in the limiting expression because the limiting kernel in Theorem \ref{singjump} is analytic in $a$ and $\bar{b}$.}  (with $y=\gamma+i\tau$)
\begin{align*}
&\lim_{\nri}\frac{K_n(e^{ia/n},e^{ib/n};\mutil)}{K_n(1,1;\mutil)}=\Gamma(2\gamma+2)\lim_{\nri}\frac{K_n(e^{ia/n},e^{ib/n};\mutil)}{n^{2\gamma+1}}\\
&\,=(2\gamma+1)\frac{_1F_1(\bar{y};2\gamma+1;-i\bar{b})_1F_1(y;2\gamma+1;ia)-\,_1F_1(1+\bar{y};2\gamma+1;-i\bar{b})_1F_1(1+y;2\gamma+1;ia)}{i(\bar{b}-a)},
\end{align*}
where the last equality is given in the proof of \cite[Theorem 5]{BNR}.  This proves the desired result when $\mu=\mutil$ and for certain values of $a$ and $b$.


To complete the proof and to prove the assertions about uniformity when $\mu=\mutil$, we will show that
\[
\left\{\frac{K_n(e^{ia/n},e^{i\bar{b}/n};\mutil)}{K_n(1,1;\mutil)}\right\}_{n\in\bbN}
\]
is a normal family in the complex variables $a$ and $b$.  The key to the calculation will be the following lemma.

\begin{lemma}\label{fbound}
If $y=\gamma+i\tau$ with $\tau\in\bbR$ and $\gamma>-1/2$, then for any compact set $K\subset \bbC$ there is a constant $C_K$ such that for any $a\in K$, $n\in\bbN$, and $m\in\{0,1,\ldots,n\}$ it holds that
\[
|\,_2F_1(-m,1+y;2\gamma+1;1-e^{ia/n})|\leq \,_2F_1(-n,1+|y|;2\gamma+1;-C_K/n).
\]
\end{lemma}

\begin{proof}
The proof is by direct calculation.  We can find a constant $C_K$ so that if $a\in K$, then
\begin{align*}
&|\,_2F_1(-m,1+y;2\gamma+1;1-e^{ia/n})|\\
&\qquad\qquad\qquad=\left|\sum_{k=0}^m\binom{m}{k}\frac{(1+y)_k}{(2\gamma+1)_k}(e^{ia/n}-1)^k\right|\leq\sum_{k=0}^m\binom{m}{k}\frac{(1+|y|)_k}{(2\gamma+1)_k}\left(\frac{C_K}{n}\right)^k\\
&\qquad\qquad\qquad\leq\sum_{k=0}^n\binom{n}{k}\frac{(1+|y|)_k}{(2\gamma+1)_k}\left(\frac{C_K}{n}\right)^k=\,_2F_1(-n,1+|y|;2\gamma+1;-C_K/n).
\end{align*}
\end{proof}

Now, we can use Lemma \ref{kn11}, Lemma \ref{fbound}, and basic properties of the $\Gamma$ function to conclude that if $a,b$ are in some compact set $K$, then there are constants $C$, $C'$, and $C_K$ so that
\begin{align*}
&\left|\frac{K_n(e^{ia/n},e^{i\bar{b}/n};\mutil)}{K_n(1,1;\mutil)}\right|\\
&\quad\leq\frac{C}{n^{2\gamma+1}}\sum_{m=0}^n\left|\,_2F_1(-m,1+y;2\gamma+1;1-e^{ia/n})\,_2F_1(-m,1+y;2\gamma+1;1-e^{ib/n})\right|\frac{\Gamma(2\gamma+1+m)}{\Gamma(m+1)}\\
&\quad\leq\frac{C'}{n}\sum_{m=0}^n\left(\,_2F_1(-n,1+|y|;2\gamma+1;-C_K/n)\right)^2\\
&\quad=C' \left(\,_2F_1(-n,1+|y|;2\gamma+1;-C_K/n)\right)^2.
\end{align*}
As $\nri$, this last quantity is bounded (see \cite[Equation (50)]{BNR}) and hence we have shown the desired normality property.  The fact that the limiting kernel that we calculated when both $a$ and $b$ are real (and distinct) is analytic in the variables $a$ and $\bar{b}$ implies the desired uniform convergence on compact subsets of $\bbC\times\bbC$.

\smallskip

We have thus proven Theorem \ref{singjump} (and the subsequent corollaries) in the special case $\mu=\mutil$.

\subsection{Hypergeometric Functions}\label{hyper}

In this section we prove some auxiliary results about hypergeometric functions and the functions that appear in our main results.  For any $a\in\bbR$, define
\begin{align}\label{wlater}
T(a):=&\bigg(\,_1F_1(1+\gamma;2\gamma+1;-ia)\,_1F_1(2+\gamma;2\gamma+2;ia)(1+\gamma)\\
\nonumber&\qquad\qquad\qquad\qquad\qquad\qquad\qquad-\,_1F_1(\gamma;2\gamma+1;-ia)\,_1F_1(1+\gamma;2\gamma+2;ia)\gamma\bigg),
\end{align}
which is the function appearing in the remark after the statement of Corollary \ref{circlesing}, but restricted to a real variable.  Our first lemma shows that $T$ never vanishes in $\bbR$.

\begin{lemma}\label{full}
If $\gamma>-1/2$, then $T(a)$ is a non-vanishing function on $\bbR$.
\end{lemma}

\begin{proof}
Let $\sigma$ be the measure $\mutil$ when $\tau=0$.  Lemma \ref{kn11} tells us that
\[
\lim_{\nri}\frac{K_n(1,1;\sigma)}{n^{2\gamma+1}}=\frac{1}{\Gamma(2\gamma+2)}>0.
\]
Suppose for contradiction that $a\in\bbR$ is such that $T(a)=0$.  Note that we know $a\neq0$.  Corollary \ref{circlesing} in the case $\mu=\sigma$ shows that
\[
\lim_{\nri}\frac{K_n(e^{ia/n},e^{ia/n};\sigma)}{K_n(1,1;\sigma)}=T(a)=0.
\]
The Cauchy-Schwarz inequality tells us that for any $b\in\bbC$, we have
\[
\frac{|K_n(e^{ia/n},e^{ib/n};\sigma)|}{K_n(1,1;\sigma)}\leq\sqrt{\frac{K_n(e^{ia/n},e^{ia/n};\sigma)}{K_n(1,1;\sigma)}}\sqrt{\frac{K_n(e^{ib/n},e^{ib/n};\sigma)}{K_n(1,1;\sigma)}}.
\]
The first square root tends to $0$ as $\nri$ and the second tends to some finite limit as $\nri$, so we have
\[
\lim_{\nri}\frac{K_n(e^{ia/n},e^{ib/n};\sigma)}{K_n(1,1;\sigma)}=0
\]
for all $b\in\bbC$.  However, this implies that the expression appearing on the right-hand side of (\ref{kform2}) is equal to $0$ for all values of $b\in\bbC$.  If we plug in $b=0$, then we see that
\[
_1F_1(\gamma;2\gamma+1;ia)=\,_1F_1(1+\gamma;2\gamma+1;ia)\neq0,
\]
where the last inequality follows from \cite[Theorem 1]{Wynn}.  Dividing through by this quantity on the right-hand side of (\ref{kform2}) shows
\[
_1F_1(\gamma;2\gamma+1;-i\bar{b})=\,_1F_1(1+\gamma;2\gamma+1;-i\bar{b})
\]
for all complex numbers $b$.  In other words, the functions
\[
_1F_1(\gamma;2\gamma+1;z)\qquad\qquad\mbox{ and }\qquad\qquad\,_1F_1(1+\gamma;2\gamma+1;z)
\]
are the same function.  However, these functions do not have the same derivative at $0$, so this gives us our contradiction.
\end{proof}

\noindent\textit{Remark.}  It would be interesting to see if one could prove Lemma \ref{full} without using reproducing kernel asymptotics.

\smallskip

Our next technical lemma will be helpful in establishing the finiteness of certain asymptotics in the next section.

\begin{lemma}\label{1f1}
For any complex number $a$ and $\gamma>-1/2$, it holds that
\begin{align}\label{zeroo}
(2\gamma+1)\frac{|\,_1F_1(\gamma;2\gamma+1;ia)|^2-|\,_1F_1(1+\gamma;2\gamma+1;ia)|^2}{2\Imag(a)}>0,
\end{align}
where if $a\in\bbR$, then we interpret this expression as $T(a)$.
\end{lemma}

\begin{proof}
Let us keep our definition of $\sigma$ from the previous proof.  We will use ideas from the proof of \cite[Proposition 12]{FMMFS}.  First, notice that when $s$ and $t$ are real, we have
\[
\overline{_1F_1(s;t;z)}=\,_1F_1(s;t;\bar{z}).
\]
Therefore, Corollary \ref{circlesing} in the case $\mu=\sigma$ implies
\begin{align}\label{fup}
\frac{|\,_1F_1(\gamma;2\gamma+1;ia)|^2-|\,_1F_1(1+\gamma;2\gamma+1;ia)|^2}{\Imag(a)}\geq0.
\end{align}
To show that this inequality is strict when $\Imag(a)\neq0$, consider the function
\[
\Theta(a):=\frac{_1F_1(1+\gamma;2\gamma+1;ia)}{\,_1F_1(\gamma;2\gamma+1;ia)}=\frac{e^{ia}\,_1F_1(\gamma;2\gamma+1;-ia)}{\,_1F_1(\gamma;2\gamma+1;ia)},
\]
where we used Kummer's transformation, which says
\[
_1F_1(x;y;z)=e^{z}\,_1F_1(y-x;y;-z)
\]
(see \cite[Equation 12a]{Buch}).  We know from \cite[Theorem 1]{Wynn}, that $\Theta(a)$ is analytic on $\{a:\Imag(a)\geq0\}$ and it is easy to see that $|\Theta(a)|=1$ when $a$ is real.  Equation (\ref{fup}) tells us that $|\Theta(a)|\leq1$ when $\Imag(a)>0$ and $\Theta(a)$ is not constant, so by the Maximum Modulus Principle, we must have $|\Theta(a)|<1$ when $\Imag(a)>0$, so this proves the claim when $\Imag(a)>0$.  We prove the claim when $\Imag(a)<0$ by considering $1/\Theta(a)$ and a similar argument.  The statement for $a\in\bbR$ is Lemma \ref{full}. 
\end{proof}

Following the notation of \cite{LubHB}, we recall that the Hermite-Biehler class of entire functions is the set of all entire functions $E$ that have no zeros in the upper half-plane $\{z:\Imag(z)>0\}$ and satisfy
\[
|E(z)|\geq|E(\bar{z})|,\qquad\qquad \Imag(z)>0.
\]
The proof of Lemma \ref{1f1} tells us that when $\gamma>-1/2$, the function $_1F_1(\gamma;2\gamma+1;iz)e^{-iz/2}$ is in the Hermite-Biehler class. A related result was proven in \cite[Proposition 12]{FMMFS}.

\subsection{The General Case}\label{generalc}

Now we will employ Lubinsky's method and prove Theorem \ref{singjump} in the general case.  The first step is to control the diagonal of the reproducing kernel, and to do this, we will use Christoffel functions.  Our next lemma concerns the measure $\mutil$ from Section \ref{firstc}.

\begin{lemma}\label{lambdatil}
If $\gamma>-1/2$, then there is a continuous function $G:\bbC\rightarrow(0,\infty)$ such that
\begin{align}\label{w2}
\lim_{\nri}n^{2\gamma+1}\lambda_n(e^{ia/n};\mutil)=G(a)
\end{align}
and the convergence is uniform on compact subsets of $\bbC$.
\end{lemma}

\noindent\textit{Remark.}  If $a$ is a real number and $\gamma=0$, then $G(a)$ can be bounded above and below by appealing to results in \cite{MNT}.

\begin{proof}
Existence and continuity of the limit (as an extended real number) follow from the calculations in Section \ref{firstc}.  Since the convergence demonstrated in Section \ref{firstc} was uniform in $a$ and $b$ in compact subsets of $\bbC$, the statement about uniformity will follow once we bound the function $G$ away from zero and infinity on compact sets.

To estimate the limit, notice that there are positive constants $C_1$ and $C_2$ so that $C_1\sigma\leq\mutil\leq C_2\sigma$, where $\sigma$ is the measure $\mutil$ but with $\tau=0$.  Therefore, due to the monotonicity properties of Christoffel functions it suffices to establish the desired bounds in the case $\tau=0$.  Lemma \ref{kn11} and Corollary \ref{circlesing} in the case $\mu=\sigma$ tell us
\begin{align*}
\lim_{\nri}n^{2\gamma+1}\lambda_n(e^{ia/n};\sigma)=\Gamma(2\gamma+2)\begin{cases}
\frac{2\Imag(a)/(2\gamma+1)}{|\,_1F_1(\gamma;2\gamma+1;ia)|^2-|\,_1F_1(1+\gamma;2\gamma+1;ia)|^2} & \qquad \Imag(a)\neq0\\
1/T(a)& \qquad \Imag(a)=0
\end{cases}
\end{align*}
The desired conclusion now follows from Lemma \ref{1f1}.
\end{proof}

Lemma \ref{lambdatil} tells us that the right-hand side of (\ref{kform0}) is never $0$ when $a=b$.  Therefore, we can strengthen our earlier observation and generalize \cite[Theorem 1]{Wynn} as follows.

\begin{corollary} 
If $\gamma>-1/2$ and $\tau\in\bbR$, then $_1F_1(\gamma+i\tau;2\gamma+1;iz)e^{-iz/2}$ is in the Hermite-Biehler class.
\end{corollary}

The next technical lemma will also be helpful.

\begin{lemma}\label{or}
For any $a\in\bbC$, any $n\in\bbN$, and any $r>0$,
\[
\left|\frac{\eitheta+e^{ia/n}}{2e^{ia/n}}\right|^{rn}\leq e^{\mco(r)},\qquad\qquad r\rightarrow0,\qquad\theta\in\bbR.
\]
Furthermore, for any compact $K\subseteq\bbC$, the implied constant can be chosen uniformly for $a\in K$, $\theta\in\bbR$, and $n\in\bbN$.
\end{lemma}

\begin{proof}
We calculate
\begin{align*}
\left|\frac{\eitheta+e^{ia/n}}{2e^{ia/n}}\right|^{rn}=\left|\frac{1+\eitheta e^{-ia/n}}{2}\right|^{rn}&=\left|\frac{1+\eitheta+\eitheta(e^{-ia/n}-1)}{2}\right|^{rn}\leq\left(1+\frac{|a|}{2n}\cdot\left|\frac{e^{-ia/n}-1}{-ia/n}\right|\right)^{rn}.
\end{align*}
Notice that $(e^{-ia/n}-1)/(-ia/n)$ is uniformly bounded in $n\in\bbN$ and in $a$ in compact subsets $K$ of $\bbC$.  If we denote this upper bound by $C_K$, then we have shown
\[
\left|\frac{\eitheta+e^{ia/n}}{2e^{ia/n}}\right|^{rn}\leq e^{C_Kr|a|/2},\qquad\qquad a\in K, \quad n\in\bbN,\quad\theta\in\bbR
\]
as desired.
\end{proof}

For $\gamma>-1/2$ and $\tau\in\bbR$ fixed, let us retain the definition of $G(a)$ from (\ref{w2}).  Lemma \ref{lambdatil} tells us that $G(a)$ is finite and non-zero for any complex number $a$ and also that $G(a)$ is a continuous function.  This allows us to proceed with the key lemma of this section.

\begin{lemma}\label{cratio}
Let $\mu$ be as in the statement of Theorem \ref{singjump} and let $\mutil$ be as in Section \ref{firstc}.  For any $a\in\bbC$ it holds that
\[
\lim_{\nri}\frac{K_n(e^{ia/n},e^{ia/n};\mutil)}{K_n(e^{ia/n},e^{ia/n};\mu)}=g(0)
\]
and the convergence is uniform for $a$ in compact subsets of $\bbC$.
\end{lemma}

\begin{proof}
We will use ideas from the proof of \cite[Theorem 7]{MNT}.  Fix $\delta\in(0,g(0))$.  Choose $\epsilon>0$ small enough so that the measure $\mu_{\epsilon}$ defined by
\[
\mu_\epsilon=\begin{cases}
\mu & \qquad \mbox{ on } \quad\{z:|z-1|>\epsilon\}\\
(g(0)+\delta)\mutil & \qquad \mbox{ on }\quad \{z:|z-1|\leq\epsilon\}
\end{cases}
\]
satisfies $\mu_\epsilon\geq\mu$.

Fix some $r\in(0,1)$ very small.  For each $m\in\bbN$, define
\[
Q_m(z):=\frac{K_m(z,e^{ia/n};\mutil)}{K_m(e^{ia/n},e^{ia/n};\mutil)},
\]
which is extremal for the problem defining $\lambda_m(e^{ia/n};\mutil)$.  We calculate
\begin{align*}
\lambda_n(e^{ia/n};\mu)&\leq\int |Q_{n-\lfloor rn\rfloor}(z)|^2\left|\frac{z+e^{ia/n}}{2e^{ia/n}}\right|^{2\lfloor rn\rfloor}d\mu(z)\\
&\leq\int_{|z-1|\leq\epsilon} |Q_{n-\lfloor rn\rfloor}(z)|^2\left|\frac{z+e^{ia/n}}{2e^{ia/n}}\right|^{2\lfloor rn\rfloor}d\mu_\epsilon+\int_{|z-1|>\epsilon} |Q_{n-\lfloor rn\rfloor}(z)|^2\left|\frac{z+e^{ia/n}}{2e^{ia/n}}\right|^{2\lfloor rn\rfloor}d\mu_\epsilon.
\end{align*}
The first of these integrals can be bounded above by
\begin{align*}
(g(0)+\delta)e^{\mco(r)}\int_{|z-1|\leq\epsilon} |Q_{n-\lfloor rn\rfloor}(z)|^2d\mutil(z)=(g(0)+\delta)e^{\mco(r)}\lambda_{n-\lfloor rn\rfloor}(e^{ia/n};\mutil),
\end{align*}
where we used Lemma \ref{or} and the big-$\mco$ notation is as $r\rightarrow0$.  Similarly, we find that for some $t_1<t_2<1$,
\begin{align*}
\int_{|z-1|>\epsilon} |Q_{n-\lfloor rn\rfloor}(z)|^2\left|\frac{z+e^{ia/n}}{2e^{ia/n}}\right|^{2\lfloor rn\rfloor}d\mu_\epsilon&\leq\mco(t_1^n)\int_{|z-1|>\epsilon}|Q_{n-\lfloor rn\rfloor}(z)|^2d\mu\\
&\leq\mco(t_1^n)\mu(\partial\bbD)\|Q_{n-\lfloor rn\rfloor}\|^2_{L^{\infty}(\partial\bbD)}\\
&\leq\mco(t_2^n),
\end{align*}
where the big-$\mco$ notation is as $\nri$ and we used the regularity of the measure $\mutil$ (in particular, the property \eqref{regprop}) in the step where we replaced $t_1$ by $t_2$.  These bounds show us that (where $r_n=\lfloor rn\rfloor/n$)
\begin{align*}
&\limsup_{\nri}n^{2\gamma+1}\lambda_n(e^{ia/n};\mu)\\
&\qquad\leq(g(0)+\delta)e^{\mco(r)}\limsup_{\nri}\left(\frac{n}{n-\lfloor rn\rfloor}\right)^{2\gamma+1}(n-\lfloor rn\rfloor)^{2\gamma+1}\lambda_{n-\lfloor rn\rfloor}(e^{ia(1-r_n)/(n-\lfloor rn\rfloor)};\mutil)\\
&\qquad=(g(0)+\delta)e^{\mco(r)}\frac{1}{(1-r)^{2\gamma+1}}G(a(1-r))
\end{align*}
by Lemma \ref{lambdatil}.  Taking the infimum over all $\delta\in(0,g(0))$ and $r\in(0,1)$ and using the continuity of $G$ gives the desired upper bound on the $\limsup$.

To lower bound the $\liminf$, we define the measure
\[
\mutil_{\epsilon}=\begin{cases}
\mutil & \qquad \mbox{ on } \quad\{z:|z-1|>\epsilon\}\\
\frac{1}{g(0)-\delta}\mu & \qquad \mbox{ on }\quad \{z:|z-1|\leq\epsilon\}
\end{cases}
\]
where $\epsilon$ is chosen small enough so that satisfies $\mutil_\epsilon\geq\mutil$.  We can then perform a  calculation similar to that above and, using the regularity of $\mu$, show that
\[
\lambda_n(e^{ia/(n-\lfloor rn\rfloor)};\mutil)\leq \frac{e^{\mco(r)}\lambda_{n-\lfloor rn\rfloor}(e^{ia/(n-\lfloor rn\rfloor)};\mu)}{g(0)-\delta}+\mco(t_3^n)
\]
as $\nri$ for some $t_3<1$.  Repeating the analysis above shows
\[
\liminf_{\nri}(n-\lfloor rn\rfloor)^{2\gamma+1}\lambda_{n-\lfloor rn\rfloor}(e^{ia/(n-\lfloor rn\rfloor)};\mu)\geq(1-r)^{2\gamma+1}e^{\mco(r)}(g(0)-\delta)G\left(\frac{a}{1-r}\right).
\]
Since every $k$ sufficiently large can be written as $n-\lfloor rn\rfloor$ for some $n$ (as was shown in the proof of \cite[Theorem 3.1]{LubBerg}), this shows
\[
\liminf_{k\rightarrow\infty}k^{2\gamma+1}\lambda_{k}(e^{ia/k};\mu)\geq(1-r)^{2\gamma+1}e^{\mco(r)}(g(0)-\delta)G\left(\frac{a}{1-r}\right).
\]
Letting $r$ and $\delta$ tend to $0$ gives the desired conclusion.

The uniformity follows from the uniformity of convergence for the measure $\mutil$ and the uniformity in the implied constants in the big-$\mco$ estimates.
\end{proof}

Notice that the Cauchy-Schwarz inequality implies
\[
\frac{|K_n(e^{ia/n},e^{ib/n};\mu)|}{K_n(1,1;\mu)}\leq\frac{\sqrt{K_n(e^{ia/n},e^{ia/n};\mu)K_n(e^{ib/n},e^{ib/n};\mu)}}{K_n(1,1;\mu)}.
\]
If we combine this observation with Lemma \ref{lambdatil} and Lemma \ref{cratio}, then we arrive at the following corollary.

\begin{corollary}\label{lambdas}
Suppose $\mu$ is as in Theorem \ref{singjump} and $a$ is a complex number.  Then
\begin{align}\label{exist1}
\lim_{\nri}n^{2\gamma+1}\lambda_n(e^{ia/n};\mu)=g(0)G(a)>0.
\end{align}
Furthermore, the collection
\[
\left\{\frac{K_n(e^{ia/n},e^{i\bar{b}/n};\mu)}{K_n(1,1;\mu)}\right\}_{n\in\bbN}
\]
is a normal family in the variables $a$ and $b$.
\end{corollary}

\noindent\textit{Remark.}  If $\gamma=a=0$, then equation (\ref{exist1}) is reminiscent of a recent result established by Danka in \cite{Danka} (see also \cite{NevTo}).  If $\tau=a=0$, then this is reminiscent of \cite[Theorem 1.1]{DanTo}.  The results in \cite{Danka,DanTo} apply to more general sets than the circle.

\medskip



\begin{proof}[Proof of Theorem \ref{singjump}]  Our proof depends on \cite[Theorem 3.10]{Bourgade}, which tells us that if we verify some technical conditions, then we can make a conclusion about the similarity of the asymptotics of $K_n(z,w;\mu)$ to those of $K_n(z,w;\mutil)$.  The first condition that we must verify is that of \textit{mutual regularity}, which means we must verify that as $\nri$
\[
\sup_{\deg(P)\leq n}\left(\frac{\int |P|^2d\mu}{\int|P|^2d\mutil}\right)^{1/n}\rightarrow1,\qquad\qquad\sup_{\deg(P)\leq n}\left(\frac{\int |P|^2d\mutil}{\int|P|^2d\mu}\right)^{1/n}\rightarrow1.
\]
This is a simple consequence of the regularity of $\mu$ and $\mutil$ because if $\epsilon>0$ is fixed, $n$ is sufficiently large, and $\deg(P)\leq n$, then
\[
\|P\|^2_{L^2(\mu)}\leq\mu(\partial\bbD)\|P^2\|_{L^{\infty}(\partial\bbD)}\leq\mu(\partial\bbD)(1+\epsilon)^n\|P\|^2_{L^2(\mutil)}
\]
and a similar string of inequalities holds if we switch the roles of $\mu$ and $\mutil$.

The second technical condition we must verify follows from Lemma \ref{lambdatil}, which shows that
\[
\lim_{r\rightarrow0^+}\limsup_{\nri}\frac{K_n(e^{ia/n},e^{ia/n};\mutil)}{K_{n-\lfloor rn\rfloor}(e^{ia/n},e^{ia/n};\mutil)}=1
\]
uniformly for $a$ in compact subsets of $\bbR$.  Therefore, we may apply the conclusion of \cite[Theorem 3.10]{Bourgade} and conclude that if $a$ and $b$ are real, then

\begin{align*}
\lim_{\nri}\frac{g(0)K_n(e^{ia/n},e^{ib/n};\mu)-K_n(e^{ia/n},e^{ib/n};\mutil)}{K_n(e^{ia/n},e^{ia/n};\mu)}=0.
\end{align*}
By Corollary \ref{lambdas}, we can rewrite this as
\begin{align*}
\lim_{\nri}\frac{g(0)K_n(e^{ia/n},e^{ib/n};\mu)-K_n(e^{ia/n},e^{ib/n};\mutil)}{K_n(1,1;\mu)}=0.
\end{align*}
We can now apply Lemma \ref{cratio} to obtain the desired conclusion when $a$ and $b$ are real numbers.

To obtain the general conclusion, we apply Corollary \ref{lambdas}, which tells us that
\[
\left\{\frac{K_n(e^{ia/n},e^{i\bar{b}/n};\mu)}{K_n(1,1;\mu)}\right\}_{n\in\bbN}
\]
is a normal family in the complex variables $a$ and $b$.  Since the limiting kernel in the model case is analytic in $a$ and $\bar{b}$ and we have already verified convergence when $a$ and $b$ are real, we have the desired uniform convergence in (\ref{kform0}) on compact subsets of $\bbC\times\bbC$.
\end{proof}


\section{Proof of Theorem \ref{bergl}}\label{bergproof}

This section contains a proof of Theorem \ref{bergl}.
Since we have already remarked that the case $m=1$ in Theorem \ref{bergl} reduces to a result that is already known, we will assume that $m\geq2$.  The procedure we follow is motivated by the approach in \cite{LubBerg} and requires some technical details presented in \cite{Islands,ToTrans}.  

\subsection{A Model Case}\label{first}

As in \cite{LubBerg}, we begin by proving Theorem \ref{bergl} in one specific case.  We will consider area measure on the region $G_{r,m}$, which we denote by $\mu_0$.  In this case, there is an exact formula for the monic orthogonal polynomials when the degree is sufficiently high.  In order to avoid this degree technicality, we require the following lemma, which follows from \cite[Theorem 1.4]{ToTrans}.

\begin{lemma}\label{toinf}
If $z_0$ is as in the statement of Theorem \ref{bergl}, then
\[
\lim_{\nri}K_n(z_0,z_0;\mu_0)=\infty.
\]
\end{lemma}




Lemma \ref{toinf} tells us that when we look at the limits of kernels, we can ignore any finite collection of terms in the numerator of $K_n(z_0+a/n,z_0+b/n;\mu_0)/K_n(z_0,z_0;\mu_0)$ without changing the limit.  Therefore, we will be able to ignore all of the polynomials to which the next proposition does not apply.

Given $r\in(0,1)$ and $m\in\bbN$ and $q>0$, let $\gamma_{r,m}^{(q)}$ be the measure defined on the unit circle by
\[
d\gamma_{r,m}^{(q)}(w)=\frac{|dw|}{|r^mw+1|^q}.
\]
Notice that this measure can be written as $|f(w)|^2|dw|$, where $f$ is analytic in a neighborhood of the closed unit disk.  The following result is Proposition 7.2 in \cite{Islands}.

\begin{prop}\label{fromislands}
With $v=2-\frac{2}{m}-\frac{2s}{m}$, the following formulas are valid for all sufficiently large values of $k$:
\begin{align*}
&\Phi_{km+m-1}(z;\mu_0)=z^{m-1}(z^m-1)^k,\\
&\Phi_{km+s}(z;\mu_0)=\frac{z^sr^{m(k+1)}}{z^m-1+r^{2m}}\left(\Phi_{k+1}\left(w;\gamma_{r,m}^{(v)}\right)-\frac{\Phi_{k+1}\left(-r^m;\gamma_{r,m}^{(v)}\right)}{\Phi_{k}\left(-r^m;\gamma_{r,m}^{(v)}\right)}\Phi_{k}\left(w;\gamma_{r,m}^{(v)}\right)\right),
\end{align*}
where $w=\frac{z^m-1}{r^m}$ and $s=0,1,2,\ldots,m-2$.
\end{prop}

Proposition \ref{fromislands} tells us that in order to calculate the asymptotics of the orthogonal polynomials for $\mu_0$, we will need to know the asymptotics of the orthogonal polynomials for $\gamma_{r,m}^{(2-\frac{2}{m}-\frac{2s}{m})}$ when $s=0,1,\ldots,m-2$.  Fortunately, the analyticity of the measure $\gamma_{r,m}^{(q)}$ for all values of $q>0$ permits such a precise description, which was given in \cite{MFMS}.  Define
\[
D(z;\gamma_{r,m}^{(q)}):=\exp\left(\frac{-q}{4\pi}\int_0^{2\pi}\log|r^m\eitheta+1|\frac{\eitheta+z}{\eitheta-z}d\theta\right),
\]
which can be analytically continued to $\{z:|z|>r^m\}$.   In fact, we can evaluate the integral using \cite[Theorem 17.17]{Rudin} and calculate
\begin{align}\label{dexact}
D(z;\gamma_{r,m}^{(q)})=\left(1+\frac{r^m}{z}\right)^{q/2},
\end{align}
where the branch cut is taken so that $D(\infty;\gamma_{r,m}^{(q)})=1$ (see also \cite[Equation 8.38]{Islands}).  The following result is part of \cite[Proposition 1]{MFMS}:

\begin{prop}\label{expconv}
Uniformly on some neighborhood of the unit circle it holds that
\[
\lim_{k\rightarrow\infty}\frac{\Phi_{k}(w;\gamma_{r,m}^{(q)})}{w^{k}}=D(w;\gamma_{r,m}^{(q)}).
\]
\end{prop}

Now that we have a precise understanding of the monic orthogonal polynomials for $\mu_0$, we need to understand the behavior of the leading coefficients $\kappa_n(\mu_0)$.  The following result is contained in the proof of \cite[Proposition 7.1]{Islands}.

\begin{prop}\label{kappone}
For every $k\in\bbN$, it holds that
\[
\kappa_{km+m-1}(\mu_0)^2=\frac{mk+m}{\pi r^{2mk+2m}}.
\]
If $s=0,\ldots,m-2$, then with $v=2-\frac{2}{m}-\frac{2s}{m}$, the following asymptotic formula holds as $\kri$:
\[
\kappa_{km+s}(\mu_0)^2=\left(\frac{2\Phi_{k}\left(-r^m;\gamma_{r,m}^{(v)}\right)(mk+1+s)}{-\Phi_{k+1}\left(-r^m;\gamma_{r,m}^{(v)}\right)r^{2km+m}}\right)\left(\frac{1}{2\pi}+\mco(\eta^{2k})\right),
\]
where $\eta$ is any real number in $(r^m,1)$.  
\end{prop}

Using results from the theory of orthogonal polynomials on the unit circle, we can improve this result in the following way:

\begin{prop}\label{newkapp}
If $s=0,\ldots,m-2$, then with $v=2-\frac{2}{m}-\frac{2s}{m}$, the following asymptotic formula holds as $\kri$:
\[
\kappa_{km+s}(\mu_0)^2=\left(\frac{(mk+1+s)(1+\frac{v}{2k})}{\pi r^{2km+2m}}\right)(1+\mco(k^{-2})).
\]
\end{prop}

\begin{proof}
Let $v=2-\frac{2}{m}-\frac{2s}{m}$.  From \cite[Equation 8.45]{Islands}, we see that
\begin{align}\label{845}
\frac{\Phi_{k+1}\left(-r^m;\gamma_{r,m}^{(v)}\right)}{\Phi_{k}\left(-r^m;\gamma_{r,m}^{(v)}\right)}=-r^m\left[1-\frac{v}{2k}+\mco(k^{-2})\right],\qquad\qquad k\rightarrow\infty.
\end{align}
Plugging this into the formula from Proposition \ref{kappone} gives the result
\end{proof}



Now we are ready to begin our calculation, which uses ideas from the proof of the main result in \cite{LubBerg}.

\begin{proof}[Proof of Theorem \ref{bergl} when $\mu=\mu_0$]
Recall the definition of $D$ from the statement of Theorem \ref{bergl} and let $J\subseteq\partial G_{r,m}\cap D$ be a closed arc containing a point $z_0$.  Let us fix some $Y_0\in m\bbN$ large enough so that the conclusion of Proposition \ref{fromislands} is valid for $\Phi_n(z;\mu_0)$ when $n\geq Y_0$.  Define $Y:=Y_0/m$, $z_x:=z_0+\frac{x}{n}$, $w_x=\frac{z_x^m-1}{r^m}$, $v_s=2-\frac{2}{m}-\frac{2s}{m}$, and suppose $n=Lm+M$, where $0\leq M<m$.  We have (as $\nri$)
\begin{align}
\nonumber&K_{Lm+M}\left(z_a,z_b;\mu_0\right)=\sum_{j=Y_0}^{Lm+M}\varphi_j\left(z_a;\mu_0\right)\overline{\varphi_j\left(z_b;\mu_0\right)}+\mco(1)\\
\nonumber&\qquad\qquad=\sum_{s=0}^{m-1}\,\sum_{k=Y}^{L-1}\varphi_{km+s}\left(z_a;\mu_0\right)\overline{\varphi_{km+s}\left(z_b;\mu_0\right)}+\mco(1)+\sum_{h=0}^M\varphi_{Lm+h}\left(z_a;\mu_0\right)\overline{\varphi_{Lm+h}\left(z_b;\mu_0\right)}\\
\nonumber&\qquad\qquad=\sum_{s=0}^{m-1}\,\sum_{k=Y}^{L-1}\kappa_{km+s}(\mu_0)^2\Phi_{km+s}\left(z_a;\mu_0\right)\overline{\Phi_{km+s}\left(z_b;\mu_0\right)}+\mco(1)\\
\nonumber&\qquad\qquad\qquad\qquad+\sum_{h=0}^M\kappa_{Lm+h}(\mu_0)^2\Phi_{Lm+h}\left(z_a;\mu_0\right)\overline{\Phi_{Lm+h}\left(z_b;\mu_0\right)}\\
\nonumber&\,=\sum_{s=0}^{m-2}\,\sum_{k=Y}^{L-1}\frac{(mk+1+s)(1+\frac{v_s}{2k})}{\pi(1+o(1))}\frac{D_s(w_a)\overline{D_s(w_b)}w_a^k\overline{w_b}^k(w_a+r^m(1+\frac{v_s}{2k}))(\overline{w_b}+r^m(1+\frac{v_s}{2k}))z_a^s\overline{z_b}^s}{(z_a^m-1+r^{2m})(\overline{z_b}^m-1+r^{2m})}\\
\label{ssplit}&\qquad\qquad\qquad+\sum_{k=0}^{L-1}\frac{m(k+1)}{\pi r^{m(2k+2)}}z_a^{m-1}\left(z_a^m-1\right)^k\overline{z_b^{m-1}\left(z_b^m-1\right)^k}+\mco(1)\\
\nonumber&\quad+\sum_{h=0}^M\frac{(Lm+1+h)(1+\frac{v_{h}}{L})}{\pi(1+o(1))}\frac{D_h(w_a)\overline{D_h(w_b)}w_a^L\overline{w_b}^L(w_a+r^m(1+\frac{v_h}{2L}))(\overline{w_b}+r^m(1+\frac{v_h}{2L}))z_a^h\overline{z_b}^h}{(z_a^m-1+r^{2m})(\overline{z_b}^m-1+r^{2m})},
\end{align}
where $D_s(x):=D(x;\gamma_{r,m}^{(2-\frac{2}{m}-\frac{2s}{m})})$.  The $\mco(1)$ term accounts for the fact that Proposition \ref{fromislands} only applies to sufficiently large indices.  The expression (\ref{ssplit}) has three parts: the double sum over $s$ and $k$ (call it $S_1$), the single sum over $k$ (call it $S_2$), and the sum over the index $h$ (call it $S_3$).  Let us estimate $S_3$ first.

In the sum $S_3$, the $o(1)$ term is as $\blri$.  Notice that the functions $D(\cdot;\gamma_{r,m}^{(v_s)})$ are uniformly bounded in a neighborhood of $\partial\bbD$.  Also, $w_a$ is within a distance of order $\mco(n^{-1})$ of the unit circle when $n$ is large and $z_a^m-1$ has modulus close to $r^m$ when $n$ is large and the same is true when $a$ is replaced by $b$ and these estimates can be made uniformly for $a$ and $b$ in compact subsets of $\bbC$ and $z_0\in J$.  Therefore, every term in $S_3$ is $\mco(L)$ as $L\rightarrow\infty$ and since there are at most $m$ such terms, we have
\begin{align}\label{s3}
S_3=\mco(L),\qquad\qquad\blri.
\end{align}

Now let us consider the sum $S_2$, which is given by
\begin{align}\label{middlesum}
\sum_{k=0}^{L-1}\frac{m(k+1)}{\pi r^{m(2k+2)}}z_a^{m-1}\left(z_a^m-1\right)^k\overline{z_b^{m-1}\left(z_b^m-1\right)^k}.
\end{align}
We can rewrite it as
\begin{align*}
\frac{mz_a^{m-1}\overline{z_b}^{m-1}}{\pi r^{2m}}\sum_{k=0}^{L-1}(k+1)\left[r^{-2m}\left(z_a^m-1\right)\left(\overline{z_b^m-1}\right)\right]^k=\frac{mz_a^{m-1}\overline{z_b}^{m-1}}{\pi r^{2m}}\sum_{k=0}^{L-1}(k+1)\left[w_a\overline{w}_b\right]^k.
\end{align*}
Recall that
\begin{align}\label{derform}
\sum_{k=0}^P(k+1)z^k=-(P+2)\frac{z^{P+1}}{1-z}+\frac{1-z^{P+2}}{(1-z)^2},\qquad\qquad z\neq1,
\end{align}
(see \cite{LubBerg}).  Using this formula, one can see that
\[
\left\{\frac{2}{L^2}\sum_{k=0}^{L-1}(k+1)\left(1+\frac{z}{L}\right)^k\right\}_{L\in\bbN}
\]
is a normal family, and as $\blri$ this sequence of functions converges to $H(z)$ uniformly on compact subsets of $\bbC$, where $H$ is defined as in the statement of Theorem \ref{bergl}.

If $z_0^m-1=w_0 r^m$ with $|w_0|=1$, then we can write
\begin{align*}
\left(z_a^m-1\right)\left(\overline{z_b^m-1}\right)&=r^{2m}+\frac{r^mm}{n}(\overline{w_0}az_0^{m-1}+w_0\bar{b}\bar{z}_0^{m-1})+\mco(n^{-2})\\
&=r^{2m}+\frac{r^m\overline{w_0}az_0^{m-1}+r^mw_0\bar{b}\bar{z}_0^{m-1}}{L}+\mco(L^{-2})
\end{align*}
as $\blri$.  If we define
\begin{align}\label{adef}
A:=\frac{\overline{w}_0az_0^{m-1}+w_0\bar{b}\bar{z}_0^{m-1}}{r^{m}},
\end{align}
then we can write
\[
\sum_{k=0}^{L-1}(k+1)\left[w_a\overline{w}_b\right]^k=\sum_{k=0}^{L-1}(k+1)\left(1+\frac{A(1+o(1))}{L}\right)^k=\frac{L^2}{2}H(A)+o(L^2)
\]
as $\blri$ and the convergence is uniform in $A$ in compact subsets of $\bbC$, and hence uniform in $z_0\in J$ and uniform in $a$ and $b$ in compact subsets of $\bbC$.  We conclude that
\begin{align}\label{middform}
S_2&=\frac{m|z_0|^{2(m-1)}}{\pi r^{2m}}\left[\frac{L^2}{2}H(A)+o(L^2)\right]
\end{align}
as $\blri$.

Now let us consider the sum $S_1$ in (\ref{ssplit}).  Due to the complexity of this expression, we will need to break it into pieces.  We begin by considering the sum only over the index $k$.  This allows us to temporarily ignore the factors that only depend on the index $s$ and consider
\begin{align*}
&\sum_{k=Y}^{L-1}\frac{(mk+1+s)(1+o(1))}{(z_a^m-1+r^{2m})(\overline{z_b}^m-1+r^{2m})}\left(1+\frac{v_s}{2k}\right)w_a^k\overline{w_b}^k\left(w_a+r^m\left(1+\frac{v_s}{2k}\right)\right)\left(\overline{w_b}+r^m\left(1+\frac{v_s}{2k}\right)\right),
\end{align*}
where the $o(1)$ term is as $k\rightarrow\infty$.  We can rewrite the sum as
\begin{align}
\nonumber&\frac{1}{r^{2m}}\sum_{k=Y}^{L-1}(mk+1+s)w_a^k\overline{w}_b^k\left(1+\frac{v_s}{2k}\right)\left(1+\frac{r^mv_s}{2k(w_a+r^m)}\right)\left(1+\frac{r^mv_s}{2k(\overline{w}_b+r^m)}\right)(1+o(1))\\
\label{simplify}&\qquad=\frac{m}{r^{2m}}\sum_{k=Y}^{L-1}kw_a^k\overline{w}_b^k+\sum_{k=Y}^{L-1}T(k,a,b,s,n,z_0),
\end{align}
where the term $T(k,a,b,s,n,z_0)$ is such that
\[
\lim_{\kri}\frac{|T(k,a,b,s,n,z_0)|}{k}=0
\]
and the convergence is uniform in $a$ and $b$ in compact subsets of the complex plane, uniform in $s\in\{0,1,\ldots,m-2\}$, uniform in $z_0\in J$, and uniform in $n\in\bbN$ sufficiently large.  Therefore, using the calculations in the evaluation of $S_2$, we find
\begin{align*}
\mbox{Equation }(\ref{simplify})&=\frac{m}{r^{2m}}\sum_{k=Y}^{L-1}kw_a^k\overline{w}_b^k+o(L^2)\\
&=\frac{L^2}{2}H(A)\frac{m}{r^{2m}}+o(L^2),\qquad\qquad\blri.
\end{align*}
Furthermore, the $o(L^2)$ term obeys these asymptotics uniformly in $a$ and $b$ in compact subsets of $\bbC$,  uniformly in $s\in\{0,1,\ldots,m-2\}$, and uniformly in $z_0\in J$.  With this in mind, we can use (\ref{dexact}) to write (as $\nri$)
\begin{align*}
\nonumber S_1&=\frac{L^2}{2}H(A)\frac{m}{\pi r^{2m}}\sum_{s=0}^{m-2}\left(1+\frac{r^m}{w_a}\right)^{1-\frac{1}{m}-\frac{s}{m}}\left(1+\frac{r^m}{\overline{w}_b}\right)^{1-\frac{1}{m}-\frac{s}{m}}z_a^s\bar{z}_b^s\left(1+o(1)\right)\\
&=\frac{L^2}{2}H(A)\frac{m}{\pi r^{2m}}\sum_{s=0}^{m-2}\left|1+\frac{r^m}{w_0}\right|^{2-\frac{2}{m}-\frac{2s}{m}}|z_0|^{2s}+o(L^2)
\end{align*}
as $\blri$, where the $o(L^2)$ term obeys these asymptotics uniformly in $a$ and $b$ in compact subsets of $\bbC$ and uniformly in $z_0\in J$.  If we write $z_0^m=1+r^me^{it}$, then $w_0=e^{it}$ and the above expression reduces to
\begin{align}\label{nosd}
S_1=\frac{L^2}{2}H(A)\frac{m}{\pi r^{2m}}\sum_{s=0}^{m-2}|z_0|^{2m-2}+o(L^2)=\frac{L^2}{2}H(A)\frac{m(m-1)}{\pi r^{2m}}|z_0|^{2m-2}+o(L^2)
\end{align}
as $\blri$.
Combining (\ref{s3}), (\ref{middform}), and (\ref{nosd}) gives us the following theorem.

\begin{theorem}\label{case1conc}
Uniformly for $a$ and $b$ in compact subsets of $\bbC$ and $z_0\in J$ it holds that
\begin{align*}
&K_{n}\left(z_0+\frac{a}{n},z_0+\frac{b}{n};\mu_0\right)=o(n^2)+n^2H(A)\frac{|z_0|^{2m-2}}{2\pi r^{2m}},
\end{align*}
as $\nri$, where $A$ is given by (\ref{adef}).
\end{theorem}

Notice that if $a=b=0$, then $A=0$.  This observation yields the following corollary, which also follows from \cite[Theorem 1.4]{ToTrans} (see (\ref{totconc}) below).

\begin{corollary}\label{case1corr}
Using the notation and hypotheses of Theorem \ref{case1conc}, uniformly for $z_0\in J$ it holds that
\begin{align*}
&K_{n}\left(z_0,z_0;\mu_0\right)=o(n^2)+n^2\frac{|z_0|^{2m-2}}{2\pi r^{2m}},
\end{align*}
as $\nri$.
\end{corollary}

Theorem \ref{case1conc} and Corollary \ref{case1corr} together prove Theorem \ref{bergl} in the case $\mu=\mu_0$.
\end{proof}

We note that for a fixed $z_0$ as in the statement of Theorem \ref{bergl}, there is an alternative proof of the uniformity of convergence in (\ref{kform3}) for $a$ and $b$ in compact subsets of $\bbC$ when $\mu=\mu_0$.  Indeed, let $G_0$ be the connected component of $\bar{G}$ that contains $z_0$ and let $\mu_{0,0}$ denote the restriction of $\mu_0$ to $G_0$.  Note that \cite[Theorem 1.1]{LubBerg} applies to the measure $\mu_{0,0}$ and the point $z_0$.  Let $Q_n$ be the extremal polynomial for the infimum defining $\lambda_n(z_0+a/n;\mu_0)$.  We calculate
\begin{align*}
\lambda_n(z_0+a/n;\mu_0)&=\int_G|Q_n(w)|^2d\mu_0(w)\geq\int_{G_0}|Q_n(w)|^2d\mu_{0,0}(w)\geq\lambda_n(z_0+a/n;\mu_{0,0})
\end{align*}
Therefore,
\begin{align*}
\frac{|K_n(z_0+a/n,z_0+b/n;\mu_0)|}{n^2}&\leq\frac{\sqrt{K_n(z_0+a/n,z_0+a/n;\mu_0)}}{n}\frac{\sqrt{K_n(z_0+b/n,z_0+b/n;\mu_0)}}{n}\\
&\leq\frac{\sqrt{K_n(z_0+a/n,z_0+a/n;\mu_{0,0})}}{n}\frac{\sqrt{K_n(z_0+b/n,z_0+b/n;\mu_{0,0})}}{n}
\end{align*}
and \cite[Theorem 1.1]{LubBerg} and \cite[Theorem 1.4]{ToTrans} show that this last quantity is uniformly bounded in $a$ and $b$ in compact subsets of $\bbC$.  Therefore,
\[
\left\{\frac{K_n(z_0+a/n,z_0+\bar{b}/n;\mu_0)}{K_n(z_0,z_0;\mu_0)}\right\}_{n\in\bbN}
\]
is a normal family in the complex variables $a$ and $b$, so convergence to a limiting function is uniform on compact subsets of $\bbC^2$.

\subsection{The General Case}\label{gencase}

The complete proof of Theorem \ref{bergl} now follows from the calculations in Section \ref{first} exactly as in Sections 3 and 4 of \cite{LubBerg}.  Indeed, many of the results of \cite[Sections 3 $\&$ 4]{LubBerg} do not depend on the simple connectivity of the support of the measure, only on the fact that $z_0$ is a peak polynomial point (as defined in \cite{LubBerg}) and a model case for comparison.  Rather than duplicate the proof in \cite{LubBerg}, we provide here only a sketch of the necessary arguments.

The first step in proving the general case is to establish estimates on the Christoffel functions, specifically on expressions of the form
\[
\lim_{\nri}n^2\lambda_n(u_n;\mu),
\]
where $u_n\in\partial G_{r,m}$ and $|u_n-u|=\mco(n^{-1})$ as $\nri$ for some $u\in\partial G_{r,m}\cap D$.  Fortunately, such a result was already proven in \cite[Theorem 1.4]{ToTrans} and tells us that as long as $u\in\partial G_{r,m}\cap D$, we have
\begin{align}\label{totconc}
\lim_{\nri}n^2\lambda_n(u_n;\mu)=2\pi \mu'(u)\left(\frac{|u|^{m-1}}{r^m}\right)^{-2},
\end{align}
where we used \cite[Equation 2.2]{ToTrans}.  The result in \cite{ToTrans} also tells us that if $M>0$ is fixed, then the conclusion (\ref{totconc}) holds uniformly for $u$ in compact subsets of $\partial G_{r,m}\cap D$ and sequences $\{u_n\}_{n\in\bbN}$ such that $|u_n-u|\leq M/n$ (because of our continuity assumption on $\mu'$).  This result will serve as an analog of \cite[Theorem 3.1]{LubBerg} for our proof.

The remainder of the proof proceeds exactly as in \cite[Section 4]{LubBerg}.  Notice that \cite[Lemma 4.1]{LubBerg} does not depend on the geometry of $\supp(\mu)$ nor does it depend on the absolute continuity of the measure; it is simply a monotonicity result that applies to the setting of Theorem \ref{bergl}.  The proof of \cite[Lemma 4.2]{LubBerg} depends only on \cite[Theorem 3.1]{LubBerg}, of which we have an appropriate analog in (\ref{totconc}).  Therefore, we may obtain the same conclusion.  The proof of \cite[Lemma 4.3]{LubBerg} also immediately carries over to our setting if we use $K_n(z,t;\mu_0)$ in place of $\tilde{K}^C_n(z,t)$.  The statement and proof of \cite[Lemma 4.4]{LubBerg} remain valid in our setting as well.  The proof of \cite[Theorem 1.1]{LubBerg} can then be copied almost verbatim to complete the proof of Theorem \ref{bergl}.

\vspace{2mm}

\noindent\textbf{Acknowledgements.}  We would like to thank the anonymous referees for several useful comments that improved the exposition of this work and especially for bringing the references \cite{Bourgade,BNR} to our attention.



\end{document}